\documentclass[12pt,a4paper,leqno]{amsart}

\usepackage[latin1]{inputenc}
\usepackage[T1]{fontenc}
\usepackage{amsfonts}
\usepackage{amsmath}
\usepackage{amssymb}
\usepackage{eurosym}
\usepackage{mathrsfs}
\usepackage{palatino}
\usepackage{color}
\usepackage{esint}
\usepackage{url}

\newcommand{\R}{\mathbb{R}}

\newcommand{\Z}{\mathbb{Z}}

\numberwithin{equation}{section}

\swapnumbers
\theoremstyle{plain}
\newtheorem{thm}[equation]{Theorem}
\newtheorem{lem}[equation]{Lemma}
\newtheorem{prop}[equation]{Proposition}
\newtheorem{cor}[equation]{Corollary}

\theoremstyle{definition}

\theoremstyle{remark}
\newtheorem{rem}[equation]{Remark}

\title{Square functions with general measures}

\author{Henri Martikainen}
\address{D\'epartement de Math\'ematiques, B\^atiment 425, Facult\'e des Sciences d'Orsay, Universit\'e Paris-Sud 11, F-91405 Orsay Cedex}
\email{henri.martikainen@math.u-psud.fr}
\thanks{H.M. is supported by the Emil Aaltonen Foundation. M.M. is supported by Fondation de Math\'ematiques Jacques Hadamard (FMJH).
The authors wish to thank Universit\'e Paris-Sud 11, Orsay, for its hospitality.}

\author{Mihalis Mourgoglou}
\address{D\'epartement de Math\'ematiques, B\^atiment 425, Facult\'e des Sciences d'Orsay, Universit\'e Paris-Sud 11, F-91405 Orsay Cedex}
\email{mihalis.mourgoglou@math.u-psud.fr}

\makeatletter
\@namedef{subjclassname@2010}{%
  \textup{2010} Mathematics Subject Classification}
\makeatother


\begin{document}
\maketitle

\begin{abstract}
We characterize the boundedness of square functions in the upper half-space with general measures. The short proof is
based on an averaging identity over good Whitney regions.
\end{abstract}

\section{Introduction}

Let $\mu$ be a Borel measure on $\R^n$. We assume that $\mu(B(x,r)) \le \lambda(x,r)$ for some $\lambda\colon \R^n \times (0,\infty) \to (0,\infty)$ satisfying that
$r \mapsto \lambda(x,r)$ is non-decreasing and $\lambda(x, 2r) \le C_{\lambda}\lambda(x,r)$ for all $x \in \R^n$ and $r > 0$. Let
\begin{displaymath}
\theta_t f(x) = \int_{\R^n} s_t(x,y)f(y)\,d\mu(y), \qquad x \in \R^n,\, t > 0,
\end{displaymath}
where $s_t$ is a kernel satisfying for some $\alpha > 0$ that
\begin{displaymath}
|s_t(x,y)| \lesssim \frac{t^{\alpha}}{t^{\alpha}\lambda(x,t) + |x-y|^{\alpha}\lambda(x, |x-y|)}
\end{displaymath}
and
\begin{displaymath}
|s_t(x,y) - s_t(x,z)| \lesssim \frac{|y-z|^{\alpha}}{t^{\alpha}\lambda(x,t) + |x-y|^{\alpha}\lambda(x, |x-y|)}
\end{displaymath}
whenever $|y-z| < t/2$. We use the $\ell^{\infty}$ metric on $\R^n$.

If $Q \subset \R^n$ is a cube with sidelength $\ell(Q)$, we define the associated Carleson box $\widehat Q = Q \times (0, \ell(Q))$.
In this note we will prove the following theorem:
\begin{thm}\label{thm:main}
Assume that there exists a function $b \in L^{\infty}(\mu)$ such that
\begin{displaymath}
\Big| \int_Q b(x) \,d\mu(x) \Big| \gtrsim \mu(Q)
\end{displaymath}
and
\begin{equation}\label{eq:car}
\iint_{\widehat Q} |\theta_t b(x)|^2\,d\mu(x) \frac{dt}{t} \lesssim \mu(3Q)
\end{equation}
for every cube $Q \subset \R^n$. Then there holds that
\begin{equation}\label{eq:SFbound}
\iint_{\R^{n+1}_+} |\theta_t f(x)|^2\,d\mu(x)\frac{dt}{t} \lesssim \|f\|_{L^2(\mu)}^2, \qquad f \in L^2(\mu).
\end{equation}
\end{thm}
\begin{cor}\label{cor:main}
If
\begin{displaymath}
\iint_{\widehat Q} |\theta_t \chi_Q(x)|^2\,d\mu(x) \frac{dt}{t} \lesssim \mu(Q)
\end{displaymath}
for every cube $Q \subset \R^n$, then \eqref{eq:SFbound} holds.
\end{cor}
To the best of our knowledge such a boundedness result was previously known only in the Lebesgue case (see \cite{CJ}, \cite{S}, \cite{Ho}, \cite{GM} for such results).
Our framework covers, as is well-known, the doubling measures, the power bounded measures ($\mu(B(x,r)) \lesssim r^m$ for some $m$), and some other additional cases of interest (see Chapter 12
of \cite{HM} for an example in the context of Calder\'on--Zygmund operators).

The proof of our result follows by first establishing an averaging equality over good dyadic Whitney regions. Such an identity is inspired
by Hyt\"onen's proof of the $A_2$ conjecture \cite{Hy}, which uses a very nice refinement of the Nazarov--Treil--Volberg method of random dyadic systems.

After this the probabilistic part of the proof ends, and we may
study just one grid establishing a uniform (in the averaging parameter) bound for these good Whitney averages. Then we expand a function $f$ in the same grid using the standard $b$-adapted martingale differences.
It is not necessary to restrict this expansion into good cubes. The rest of the proof is a non-homogeneous $Tb$ type summing argument (see e.g. \cite{NTV} and \cite{HM}), which, in this setting, we manage to perform
in a delightfully clear way. Indeed, it only takes a few pages. We find that the proof is of interest, since it is, in particular, a very accessible application of the most recent non-homogeneous methods.

\begin{rem}
The property $\lambda(x, |x-y|) \sim \lambda(y, |x-y|)$ can be assumed without loss of generality.
Indeed, in Proposition 1.1 of \cite{HYY} it is shown that
$\Lambda(x,r) := \inf_{z \in \R^n} \lambda(z, r + |x-z|)$ satisfies that $r \mapsto \Lambda(x,r)$ is non-decreasing, $\Lambda(x,2r) \le C_{\lambda}\Lambda(x,r)$,
$\mu(B(x,r)) \le \Lambda(x,r)$, $\Lambda(x,r) \le \lambda(x,r)$ and $\Lambda(x,r) \le C_{\lambda}\Lambda(y,r)$ if $|x-y| \le r$. Therefore, we may (and do) assume that the dominating function $\lambda$ satisfies the additional symmetry property
$\lambda(x,r) \le C\lambda(y,r)$ if $|x-y| \le r$.
\end{rem} 
 \begin{rem}
The condition \eqref{eq:car} is necessary for \eqref{eq:SFbound} to hold. Indeed, one writes $b = b\chi_{3Q} + b\chi_{(3Q)^c}$ and notices that in \eqref{eq:car} 
the part with $b\chi_{3Q}$ is dominated by $\|b\chi_{3Q}\|_{L^2(\mu)}^2 \lesssim \mu(3Q)$, if one assumes \eqref{eq:SFbound}. For the other part, we note that
for every $x \in Q$ there holds that
\begin{align*}
 |\theta_t (b\chi_{(3Q)^c}) (x)| &\lesssim \int_{(3Q)^c} \frac{t^{\alpha}}{|x-y|^{\alpha}\lambda(x, |x-y|)}\,d\mu(y) \\
&\lesssim t^{\alpha} \int_{Q^c} \frac{|y-c_Q|^{-\alpha}}{\lambda(c_Q, |y-c_Q|)} \,d\mu(y) \lesssim t^{\alpha}\ell(Q)^{-\alpha}.
\end{align*}
This implies that
\begin{displaymath}
\iint_{\widehat Q} |\theta_t (b\chi_{(3Q)^c})(x)|^2\,d\mu(x) \frac{dt}{t} \lesssim \mu(Q) \cdot \ell(Q)^{-2\alpha} \int_0^{\ell(Q)} t^{2\alpha-1}\,dt \lesssim \mu(Q).
\end{displaymath}

The assumption of Corollary \ref{cor:main} is also necessary. However, even there one may weaken the assumption by replacing on the right-hand side $\mu(Q)$ with, say, $\mu(3Q)$ (note that
Theorem \ref{thm:main} is true with $\mu(3Q)$ replaced by $\mu(\kappa Q)$, $\kappa > 1$).
\end{rem}

\section{Proof of the main theorem}
\subsection{A random dyadic grid}
Let us be given a random dyadic grid $\mathcal{D} = \mathcal{D}(w)$, $w = (w_i)_{i \in \Z} \in (\{0,1\}^n)^{\Z}$.
This means that $\mathcal{D} = \{Q + \sum_{i:\, 2^{-i} < \ell(Q)} 2^{-i}w_i: \, Q \in \mathcal{D}_0\} = \{Q + w: \, Q \in \mathcal{D}_0\}$, where we simply have defined
$Q + w := Q + \sum_{i:\, 2^{-i} < \ell(Q)} 2^{-i}w_i$. Here $\mathcal{D}_0$ is the standard dyadic grid of $\R^n$.

We set $\gamma = \alpha/(2d+2\alpha)$, where $\alpha > 0$ appears in the kernel estimates and $d :=  \log_2 C_{\lambda}$.
A cube $Q \in \mathcal{D}$ is called bad if there exists another cube $\tilde Q \in \mathcal{D}$ so that $\ell(\tilde Q) \ge 2^r \ell(Q)$ and $d(Q, \partial \tilde Q) \le \ell(Q)^{\gamma}\ell(\tilde Q)^{1-\gamma}$.
Otherwise it is good.
One notes that $\pi_{\textrm{good}} := \mathbb{P}_{w}(Q + w \textrm{ is good})$ is independent of $Q \in \mathcal{D}_0$. The parameter $r$ is a fixed constant so large that $\pi_{\textrm{good}} > 0$
and $2^{r(1-\gamma)} \ge 3$.

Furthermore, it is important to note that for a fixed $Q \in \mathcal{D}_0$
the set $Q + w$ depends on $w_i$ with $2^{-i} < \ell(Q)$, while the goodness (or badness) of $Q + w$ depends on $w_i$ with $2^{-i} \ge \ell(Q)$. In particular, these notions are independent.

\subsection{Averaging over good Whitney regions}
Let $f \in L^2(\mu)$.
For $R \in \mathcal{D}$, let $W_R = R \times (\ell(R)/2, \ell(R))$ be the associated Whitney region.
We can assume that $w$ is such that $\mu(\partial R) = 0$ for every $R \in \mathcal{D} = \mathcal{D}(w)$ (this is the case
for a.e. $w$). Using that $\pi_{\textrm{good}} := \mathbb{P}_{w}(R + w \textrm{ is good}) = E_w \chi_{\textup{good}}(R+w)$
for any $R \in \mathcal{D}_0$ we may now write
\begin{align*}
\iint_{\R^{n+1}_+} |\theta_t f(x)|^2\,d\mu(x)\frac{dt}{t} &= E_w \sum_{R \in \mathcal{D}} \iint_{W_R} |\theta_t f(x)|^2\,d\mu(x)\frac{dt}{t} \\
& = E_w \sum_{R \in \mathcal{D}_0} \iint_{W_{R+w}} |\theta_t f(x)|^2\,d\mu(x)\frac{dt}{t} \\
&= \frac{1}{\pi_{\textrm{good}}}   \sum_{R \in \mathcal{D}_0} \pi_{\textrm{good}}  E_w \iint_{W_{R+w}} |\theta_t f(x)|^2\,d\mu(x)\frac{dt}{t} \\
&= \frac{1}{\pi_{\textrm{good}}}   \sum_{R \in \mathcal{D}_0} E_w[\chi_{\textup{good}}(R+w)]  E_w \iint_{W_{R+w}} |\theta_t f(x)|^2\,d\mu(x)\frac{dt}{t} \\
&= \frac{1}{\pi_{\textrm{good}}}   \sum_{R \in \mathcal{D}_0} E_w \Big[ \chi_{\textup{good}}(R+w) \iint_{W_{R+w}} |\theta_t f(x)|^2\,d\mu(x)\frac{dt}{t} \Big] \\
&= \frac{1}{\pi_{\textrm{good}}} E_w \sum_{R \in \mathcal{D}_{\textup{good}}} \iint_{W_R} |\theta_t f(x)|^2\,d\mu(x)\frac{dt}{t}.
\end{align*}
Notice that we used the independence of $\chi_{\textup{good}}(R+w)$ and $\iint_{W_{R+w}} |\theta_t f(x)|^2\,d\mu(x)\frac{dt}{t}$ for a
fixed $R \in \mathcal{D}_0$. In \cite{Hy}, Hyt\"onen used averaging equalities to represent a general Calder\'on--Zygmund operator as an average of dyadic shifts. These
techniques are similar in spirit.

We now fix one $w$. It is enough to show that
\begin{equation}\label{eq:red}
\mathop{\sum_{R \in \mathcal{D}_{\textup{good}}}}_{\ell(R) \le 2^s} \iint_{W_R} |\theta_t f(x)|^2\,d\mu(x)\frac{dt}{t} \lesssim \|f\|_{L^2(\mu)}^2
\end{equation}
with every large $s \in \Z$. Let us now fix the $s$ as well. 

\subsection{Adapted decomposition of $f$}
We now perform the standard $b$-adapted martingale difference decomposition of $f$.
We define $\langle f \rangle_Q = \mu(Q)^{-1}\int_Q f\,d\mu$,
\begin{displaymath}
E_Q = \frac{\langle f \rangle_Q}{\langle b \rangle_Q} \chi_Q b
\end{displaymath}
and
\begin{displaymath}
\Delta_Q f = \sum_{Q' \in \textup{ch}(Q)} \Big[\frac{\langle f \rangle_{Q'}}{\langle b \rangle_{Q'}} -  \frac{\langle f \rangle_Q}{\langle b \rangle_Q}\Big]\chi_{Q'}b.
\end{displaymath}
We can write in $L^2(\mu)$ that
\begin{displaymath}
f = \mathop{\sum_{Q \in \mathcal{D}}}_{\ell(Q) \le 2^s} \Delta_Q f  + \mathop{\sum_{Q \in \mathcal{D}}}_{\ell(Q) = 2^s} E_Q f.
\end{displaymath}
There also holds that
\begin{displaymath}
\|f\|_{L^2(\mu)}^2 \sim \mathop{\sum_{Q \in \mathcal{D}}}_{\ell(Q) \le 2^s} \|\Delta_Q f\|_{L^2(\mu)}^2  + \mathop{\sum_{Q \in \mathcal{D}}}_{\ell(Q) = 2^s} \|E_Q f\|_{L^2(\mu)}^2.
\end{displaymath}

We plug this decomposition into \eqref{eq:red} noting that we need to prove that
\begin{displaymath}
\mathop{\sum_{R \in \mathcal{D}_{\textup{good}}}}_{\ell(R) \le 2^s} \iint_{W_R} \Big| \mathop{\sum_{Q \in \mathcal{D}}}_{\ell(Q) \le 2^s}
\theta_t \Delta_Q f(x)\Big|^2\,d\mu(x)\frac{dt}{t} \lesssim \|f\|_{L^2(\mu)}^2,
\end{displaymath}
where we abuse notation by redefining the operator $\Delta_Q$ to be $\Delta_Q + E_Q$ if $\ell(Q) = 2^s$.

\subsection{The case $\ell(Q) < \ell(R)$}
Here we show that
\begin{displaymath}
\sum_{R:\, \ell(R) \le 2^s} \iint_{W_R} \Big| \sum_{Q:\, \ell(Q) < \ell(R)} \theta_t \Delta_Q f(x)\Big|^2\,d\mu(x)\frac{dt}{t} \lesssim \|f\|_{L^2(\mu)}^2.
\end{displaymath}
Let us set
\begin{displaymath}
A_{QR} = \frac{\ell(Q)^{\alpha/2}\ell(R)^{\alpha/2}}{D(Q,R)^{\alpha}\sup_{z \in Q \cup R} \lambda(z, D(Q,R))}\mu(Q)^{1/2}\mu(R)^{1/2},
\end{displaymath}
where $D(Q,R) = \ell(Q) + \ell(R) + d(Q,R)$. Notice that $A_{QR} = A_{RQ}$. We record the following proposition (this is Proposition 6.3 of \cite{HM}):
\begin{prop}\label{prop:Abound}
There holds that
\begin{displaymath}
\sum_{Q, R} A_{QR} x_{Q} y_{R} \lesssim \Big( \sum_{Q} x_{Q}^2 \Big)^{1/2}  \Big( \sum_{R} y_{R}^2 \Big)^{1/2}
\end{displaymath}
for $x_{Q}, y_{R} \ge 0$. In particular, there holds that
\begin{displaymath}
\Big( \sum_{R} \Big[ \sum_Q A_{QR} x_{Q} \Big]^2\Big)^{1/2} \lesssim \Big( \sum_{Q} x_{Q}^2 \Big)^{1/2}.
\end{displaymath}
\end{prop}

Notice that $\ell(Q) < \ell(R) \le 2^s$ implies that $\int \Delta_Q f\,d\mu = 0$. Let $(x,t) \in W_R$. We write 
\begin{align*}
 \theta_t \Delta_Q f(x) = \int_Q [s_t(x,y) - s_t(x, c_Q)]\Delta_Qf(y)\,d\mu(y),
\end{align*}
where $c_Q$ is the center of the cube $Q$.
Noting that $|y-c_Q| \le \ell(Q)/2 \le \ell(R)/4 < t/2$ for every $y \in Q$, we may estimate
\begin{align*}
|s_t(x,y) - s_t(x, c_Q)| &\lesssim \frac{\ell(Q)^{\alpha}}{\ell(R)^{\alpha}\lambda(x,\ell(R)) + d(Q,R)^{\alpha}\lambda(x, d(Q,R))} \\
&\lesssim  \frac{\ell(Q)^{\alpha/2}\ell(R)^{\alpha/2}}{D(Q,R)^{\alpha}\sup_{z \in Q \cup R} \lambda(z, D(Q,R))}.
\end{align*}
The last estimate is seen as follows: In the numerator, simply estimate $\ell(Q)^{\alpha} \le \ell(Q)^{\alpha/2}\ell(R)^{\alpha/2}$.
In the denominator we split into two cases. If $d(Q,R) \le \ell(R)$ one has $D(Q,R) \lesssim \ell(R)$, while in the case
$d(Q,R) > \ell(R)$ one has $D(Q,R) \lesssim d(Q,R)$. It remains to note that if $z \in Q \cup R$, then
$|x-z| \lesssim D(Q,R)$.

We conclude that
\begin{align*}
|\theta_t \Delta_Q f(x)| \lesssim \frac{\ell(Q)^{\alpha/2}\ell(R)^{\alpha/2}}{D(Q,R)^{\alpha}\sup_{z \in Q \cup R} \lambda(z, D(Q,R))} \mu(Q)^{1/2} \|\Delta_Qf\|_{L^2(\mu)}, \qquad (x,t) \in W_R.
\end{align*}
This yields that
\begin{align*}
\sum_{R:\, \ell(R) \le 2^s} \iint_{W_R} \Big| \sum_{Q:\, \ell(Q) < \ell(R)} \theta_t \Delta_Q f(x)\Big|^2\,d\mu(x)\frac{dt}{t} &\lesssim \sum_R \Big[  \sum_Q A_{QR} \|\Delta_Qf\|_{L^2(\mu)} \Big]^2 \\
&\lesssim \sum_Q \|\Delta_Q f\|_{L^2(\mu)}^2 \lesssim \|f\|_{L^2(\mu)}^2.
\end{align*}

\subsection{The case $\ell(Q) \ge \ell(R)$ and $d(Q,R) > \ell(R)^{\gamma}\ell(Q)^{1-\gamma}$}\label{sec:sep}
In this subsection we deal with
\begin{displaymath}
\sum_{R:\, \ell(R) \le 2^s} \iint_{W_R} \Big| \mathop{\sum_{Q:\, \ell(R) \le \ell(Q) \le 2^s}}_{d(Q,R) > \ell(R)^{\gamma}\ell(Q)^{1-\gamma}} \theta_t \Delta_Q f(x)\Big|^2\,d\mu(x)\frac{dt}{t}.
\end{displaymath}
Let $(x,t) \in W_R$. The size estimate gives that
\begin{displaymath}
|\theta_t \Delta_Q f(x)| \lesssim \int_Q \frac{\ell(R)^{\alpha}}{d(Q,R)^{\alpha}\lambda(x, d(Q,R))} |\Delta_Q f(y)|\,d\mu(y),
\end{displaymath}
where we claim that
\begin{displaymath}
 \frac{\ell(R)^{\alpha}}{d(Q,R)^{\alpha}\lambda(x, d(Q,R))} \lesssim \frac{\ell(Q)^{\alpha/2}\ell(R)^{\alpha/2}}{D(Q,R)^{\alpha}\lambda(x, D(Q,R))}.
\end{displaymath}
This estimate is trivial if $d(Q,R) \ge \ell(Q)$, because in that case $D(Q,R) \lesssim d(Q,R)$.
So assume that $d(Q,R) < \ell(Q)$, in which case $D(Q,R) \lesssim \ell(Q)$.
This case is more tricky. Note that
\begin{align*}
\lambda(x,\ell(Q)) &= \lambda(x, (\ell(Q)/\ell(R))^{\gamma} \ell(R)^{\gamma}\ell(Q)^{1-\gamma}) \\
&\lesssim C_{\lambda}^{\log_2 (\frac{\ell(Q)}{\ell(R)})^{\gamma}} \lambda(x, \ell(R)^{\gamma}\ell(Q)^{1-\gamma}) =  \Big(\frac{\ell(Q)}{\ell(R)}\Big)^{\gamma d}\lambda(x, \ell(R)^{\gamma}\ell(Q)^{1-\gamma}).
\end{align*}
Using the assumption $d(Q,R) > \ell(R)^{\gamma}\ell(Q)^{1-\gamma}$ and the identity $\gamma d + \gamma \alpha = \alpha/2$, we conclude that
\begin{displaymath}
\frac{\ell(R)^{\alpha}}{d(Q,R)^{\alpha}\lambda(x, d(Q,R))} \lesssim \frac{\ell(Q)^{\alpha/2}\ell(R)^{\alpha/2}}{\ell(Q)^{\alpha}\lambda(x, \ell(Q))} \lesssim \frac{\ell(Q)^{\alpha/2}\ell(R)^{\alpha/2}}{D(Q,R)^{\alpha}\lambda(x, D(Q,R))}.
\end{displaymath}

Noting again that if $z \in Q \cup R$, then $|x-z| \lesssim D(Q,R)$, we have shown that
\begin{align*}
|\theta_t \Delta_Q f(x)| \lesssim \frac{\ell(Q)^{\alpha/2}\ell(R)^{\alpha/2}}{D(Q,R)^{\alpha}\sup_{z \in Q \cup R} \lambda(z, D(Q,R))} \mu(Q)^{1/2} \|\Delta_Qf\|_{L^2(\mu)}, \qquad (x,t) \in W_R.
\end{align*}
This is enough by Proposition \ref{prop:Abound} like in the previous subsection.

\subsection{The case $\ell(R) \le \ell(Q) \le 2^r\ell(R)$ and $d(Q,R) \le \ell(R)^{\gamma}\ell(Q)^{1-\gamma}$}
Here we bound
\begin{align*}
\sum_{R:\, \ell(R) \le 2^s} & \iint_{W_R} \Big| \mathop{\sum_{Q:\, \ell(R) \le \ell(Q) \le \min(2^s, 2^r\ell(R))}}_{d(Q,R) \le \ell(R)^{\gamma}\ell(Q)^{1-\gamma}} \theta_t \Delta_Q f(x)\Big|^2\,d\mu(x)\frac{dt}{t} \\
&\lesssim  \sum_Q \sum_{R:\, R \sim Q}  \iint_{W_R} |\theta_t \Delta_Q f(x)|^2\,d\mu(x)\frac{dt}{t},
\end{align*}
where we have written $Q \sim R$ to mean $\ell(Q) \sim \ell(R)$ and $d(Q,R) \lesssim \min(\ell(Q), \ell(R))$.
We also used the fact that given $R$ there are $\lesssim 1$ cubes $Q$ for which $Q \sim R$.

Let $(x,t) \in W_R$. The size estimate gives that
\begin{align*}
|\theta_t \Delta_Q f(x)| &\lesssim \frac{\mu(Q)^{1/2}}{\lambda(x,t)^{1/2}} \frac{1}{\lambda(x,t)^{1/2}} \|\Delta_Q f\|_{L^2(\mu)} \\
&\lesssim  \frac{\mu(Q)^{1/2}}{\lambda(c_Q,\ell(Q))^{1/2}} \frac{1}{\lambda(c_R,\ell(R))^{1/2}} \|\Delta_Q f\|_{L^2(\mu)} \le \mu(R)^{-1/2}\|\Delta_Q f\|_{L^2(\mu)}.
\end{align*}
Therefore, we have that
\begin{displaymath}
\iint_{W_R} |\theta_t \Delta_Q f(x)|^2\,d\mu(x)\frac{dt}{t} \lesssim \|\Delta_Q f\|_{L^2(\mu)}^2,
\end{displaymath}
and so
\begin{displaymath}
\sum_Q \sum_{R:\, R \sim Q} \iint_{W_R} |\theta_t \Delta_Q f(x)|^2\,d\mu(x)\frac{dt}{t} \lesssim \sum_Q  \|\Delta_Q f\|_{L^2(\mu)}^2  \sum_{R:\, R \sim Q} 1 \lesssim \|f\|_{L^2(\mu)}^2. 
\end{displaymath}

\subsection{The case $\ell(Q) > 2^r\ell(R)$ and $d(Q,R) \le \ell(R)^{\gamma}\ell(Q)^{1-\gamma}$}
We finally utilize the goodness of $R$ to conclude that in this case we must actually have that $R \subset Q$.
This means that
\begin{align*}
\mathop{\sum_{R \in \mathcal{D}_{\textup{good}}}}_{\ell(R) \le 2^s} \iint_{W_R}& \Big| \mathop{\sum_{Q:\, 2^r\ell(R) < \ell(Q) \le 2^s}}_{d(Q,R) \le \ell(R)^{\gamma}\ell(Q)^{1-\gamma}} \theta_t \Delta_Q f(x)\Big|^2\,d\mu(x)\frac{dt}{t} \\
&= \mathop{\sum_{R \in \mathcal{D}_{\textup{good}}}}_{\ell(R) < 2^{s-r}} \iint_{W_R} \Big| \sum_{k=r+1}^{s+\textup{gen}(R)} \theta_t \Delta_{R^{(k)}} f(x)\Big|^2\,d\mu(x)\frac{dt}{t},
\end{align*}
where gen$(R)$ is determined by $\ell(R) = 2^{-\textup{gen}(R)}$, and $R^{(k)} \in \mathcal{D}$ is the unique cube for which
$\ell(R^{(k)}) = 2^k\ell(R)$ and $R \subset R^{(k)}$. We decompose
\begin{displaymath}
\Delta_{R^{(k)}} f = -B_{R^{(k-1)}} \chi_{\R^n \setminus R^{(k-1)}}b + \mathop{\sum_{S \in \textup{ch}(R^{(k)})}}_{S \ne R^{(k-1)}} \chi_{S}\Delta_{R^{(k)}} f +B_{R^{(k-1)}}b,
\end{displaymath}
where
\begin{displaymath}
B_{R^{(k-1)}} = \langle \Delta_{R^{(k)}} f / b \rangle_{R^{(k-1)}} = \left\{ \begin{array}{ll}
\frac{\langle f \rangle_{R^{(k-1)}}}{\langle b \rangle_{R^{(k-1)}}} -  \frac{\langle f \rangle_{R^{(k)}}}{\langle b \rangle_{R^{(k)}}}, & \textup{if } r+1 \le k < s+\textup{gen}(R), \\
\frac{\langle f \rangle_{R^{(k-1)}}}{\langle b \rangle_{R^{(k-1)}}}, & k = s+\textup{gen}(R). \end{array} \right.
\end{displaymath}
Noticing that
\begin{displaymath}
\sum_{k=r+1}^{s+\textup{gen}(R)} B_{R^{(k-1)}} = \frac{\langle f \rangle_{R^{(r)}}}{\langle b \rangle_{R^{(r)}}},
\end{displaymath}
we have that $\sum_{k=r+1}^{s+\textup{gen}(R)} \theta_t \Delta_{R^{(k)}} f $ equals
\begin{displaymath}
 - \sum_{k=r+1}^{s+\textup{gen}(R)} B_{R^{(k-1)}} \theta_t(\chi_{\R^n \setminus R^{(k-1)}}b)
+ \sum_{k=r+1}^{s+\textup{gen}(R)} \mathop{\sum_{S \in \textup{ch}(R^{(k)})}}_{S \ne R^{(k-1)}} \theta_t(\chi_{S}\Delta_{R^{(k)}} f)
+ \frac{\langle f \rangle_{R^{(r)}}}{\langle b \rangle_{R^{(r)}}}\theta_t b.
\end{displaymath}

Let us first deal with the last term. We bound
\begin{align*}
\mathop{\sum_{R \in \mathcal{D}_{\textup{good}}}}_{\ell(R) < 2^{s-r}} & \frac{|\langle f \rangle_{R^{(r)}}|^2}{|\langle b \rangle_{R^{(r)}}|^2} \iint_{W_R}  |\theta_t b(x)|^2\,d\mu(x)\frac{dt}{t} \\
&\lesssim \mathop{\sum_{R \in \mathcal{D}_{\textup{good}}}}_{\ell(R) < 2^{s-r}} |\langle f \rangle_{R^{(r)}}|^2 \iint_{W_R}  |\theta_t b(x)|^2\,d\mu(x)\frac{dt}{t} \\
&\le \mathop{\sum_{S \in \mathcal{D}}} |\langle f \rangle_S|^2 \mathop{\sum_{R \in \mathcal{D}_{\textup{good}}}}_{S = R^{(r)}} \iint_{W_R}  |\theta_t b(x)|^2\,d\mu(x)\frac{dt}{t} =:  \mathop{\sum_{S \in \mathcal{D}}} |\langle f \rangle_S|^2 a_S
\lesssim \|f\|_{L^2(\mu)}^2,
\end{align*}
where the last estimate follows from Carleson embedding theorem and the next lemma.
\begin{lem}
The sequence $(a_S)_{S \in \mathcal{D}}$ satisfies the Carleson condition i.e. there holds that
\begin{displaymath}
\sum_{S \subset R} a_S \lesssim \mu(R)
\end{displaymath}
for every $R \in \mathcal{D}$.
\end{lem}
\begin{proof}
Fix $R \in \mathcal{D}$, and let $\mathcal{F}(R)$ denote the maximal $Q \in \mathcal{D}$ such that $\ell(Q) \le 2^{-r}\ell(R)$ and $d(Q,R^c) \ge 3\ell(Q)$.
Notice that
\begin{align*}
\sum_{S \subset R} a_S &= \sum_{S \subset R} \mathop{\sum_{Q \in \mathcal{D}_{\textup{good}}}}_{S = Q^{(r)}} \iint_{W_Q}  |\theta_t b(x)|^2\,d\mu(x)\frac{dt}{t} \\
&\le \mathop{\mathop{\sum_{Q \in \mathcal{D}}}_{\ell(Q) \le 2^{-r}\ell(R)}}_{d(Q,R^c) \ge 3\ell(Q)} \iint_{W_Q}  |\theta_t b(x)|^2\,d\mu(x)\frac{dt}{t} \\
&= \sum_{Q \in \mathcal{F}(R)} \sum_{\tilde Q \subset Q}  \iint_{W_{\tilde Q}}  |\theta_t b(x)|^2\,d\mu(x)\frac{dt}{t} \\
&\le \sum_{Q \in \mathcal{F}(R)} \iint_{\widehat Q} |\theta_t b(x)|^2\,d\mu(x)\frac{dt}{t} \lesssim \sum_{Q \in \mathcal{F}(R)} \mu(3Q) \lesssim \mu(R),
\end{align*}
where we used goodness, the fact that $2^{r(1-\gamma)} \ge 3$, our assumption about $b$ and the fact that $\sum_{Q \in \mathcal{F}(R)} \chi_{3Q} \lesssim \chi_R$.
\end{proof}

To complete the proof of our main theorem, it remains to control
\begin{displaymath}
\mathop{\sum_{R \in \mathcal{D}_{\textup{good}}}}_{\ell(R) < 2^{s-r}} \iint_{W_R}
\Big| \sum_{k=r+1}^{s+\textup{gen}(R)} B_{R^{(k-1)}} \theta_t(\chi_{\R^n \setminus R^{(k-1)}}b)(x) \Big|^2 \,d\mu(x)\frac{dt}{t}
\end{displaymath}
and
\begin{displaymath}
\mathop{\sum_{R \in \mathcal{D}_{\textup{good}}}}_{\ell(R) < 2^{s-r}} \iint_{W_R}
\Big| \sum_{k=r+1}^{s+\textup{gen}(R)} \mathop{\sum_{S \in \textup{ch}(R^{(k)})}}_{S \ne R^{(k-1)}} \theta_t(\chi_{S}\Delta_{R^{(k)}} f)(x) \Big|^2\,d\mu(x)\frac{dt}{t}.
\end{displaymath}

By the accretivity condition for $b$, there holds that
\begin{displaymath}
|B_{R^{(k-1)}}| \lesssim \mu(R^{(k-1)})^{-1/2}\| \Delta_{R^{(k)}} f\|_{L^2(\mu)}.
\end{displaymath}
Let $(x,t) \in W_R$. The size estimate gives that
\begin{align*}
|\theta_t(\chi_{\R^n \setminus R^{(k-1)}}b)(x)| &\lesssim \ell(R)^{\alpha} \int_{\R^n \setminus B(x, d(R, \R^n \setminus R^{(k-1)}))} \frac{|x-y|^{-\alpha}}{\lambda(x,|x-y|)}\,d\mu(y) \\
&\lesssim \ell(R)^{\alpha}d(R, \R^n \setminus R^{(k-1)})^{-\alpha} \lesssim 2^{-\alpha k /2},
\end{align*}
where goodness was used to conclude that $d(R, \R^n \setminus R^{(k-1)}) \ge \ell(R)^{1/2}\ell(R^{(k-1)})^{1/2}$.

Let then $S \in \textup{ch}(R^{(k)})$, $S \subset R^{(k)} \setminus R^{(k-1)}$. Notice that $d(R, S) \ge \ell(R)^{\gamma}\ell(S)^{1-\gamma}$.
Therefore, an estimate like in the subsection \ref{sec:sep} gives that
\begin{displaymath}
|\theta_t(\chi_{S}\Delta_{R^{(k)}} f)(x)| \lesssim 2^{-\alpha k /2} \mu(R^{(k-1)})^{-1/2} \|\Delta_{R^{(k)}} f\|_{L^2(\mu)}.
\end{displaymath}

What we need then readily follows from the following estimate:
\begin{align*}
&\sum_{R:\, \ell(R) < 2^{s-r}} \mu(R) \Big[ \sum_{k=r+1}^{s+\textup{gen}(R)} 2^{-\alpha k /2} \mu(R^{(k-1)})^{-1/2} \|\Delta_{R^{(k)}} f\|_{L^2(\mu)} \Big]^2 \\
&\lesssim \sum_{R:\, \ell(R) < 2^{s-r}}  \mu(R)   \sum_{k=r+1}^{s+\textup{gen}(R)} 2^{-\alpha k /2} \mu(R^{(k-1)})^{-1} \|\Delta_{R^{(k)}} f\|_{L^2(\mu)}^2 \\
&= \sum_{k=r+1}^{\infty} 2^{-\alpha k /2} \sum_{m=k-s}^{\infty} \sum_{S:\, \ell(S) = 2^{k-m-1}} \|\Delta_{S^{(1)}} f\|_{L^2(\mu)}^2 \mu(S)^{-1} \mathop{\sum_{R:\,\ell(R) = 2^{-m}}}_{R \subset S} \mu(R) \\
&= \sum_{k=r+1}^{\infty} 2^{-\alpha k /2} \sum_{m=k-s}^{\infty} \sum_{S:\, \ell(S) = 2^{k-m-1}} \|\Delta_{S^{(1)}} f\|_{L^2(\mu)}^2 \\
&\lesssim  \sum_{k=r+1}^{\infty} 2^{-\alpha k /2} \sum_{m=k-s}^{\infty} \sum_{S:\, \ell(S) = 2^{k-m}} \|\Delta_S f\|_{L^2(\mu)}^2 \lesssim \sum_{S:\, \ell(S) \le 2^s} \|\Delta_S f\|_{L^2(\mu)}^2  \lesssim \|f\|_{L^2(\mu)}^2.
\end{align*}

\end{document}